\documentclass[11pt,english]{article}
\usepackage{amsmath}
\usepackage{amssymb}
\usepackage{amsthm}
\usepackage[english]{babel}
\usepackage{times}
\usepackage[T1]{fontenc}
\usepackage[latin1]{inputenc}
\usepackage[all]{xy}
\usepackage{geometry}

\newtheorem{teo}{Theorem}[section]

\newtheorem{oss}[teo]{Remark}

\newtheorem{defn}[teo]{Definition}
\newtheorem{lem}[teo]{Lemma}
\newtheorem{pro}[teo]{Proposition}
\newcommand{\D}{\mathbb{D}}
\newcommand{\B}{\mathbb{B}}
\newcommand{\HH}{\mathbb{H}}
\newcommand{\N}{\mathbb{N}}

\newcommand{\rr}{\mathbb{R}}
\newcommand{\s}{\mathbb{S}}

\newcommand{\de}{\partial_c}
\newcommand{\p}{\partial}

\DeclareMathOperator{\Diam}{diam}

\DeclareMathOperator{\IIm}{Im}
\DeclareMathOperator{\RRe}{Re}
\DeclareMathOperator{\prodstar}{\prod \hskip -2.30ex * \hskip 1.3 ex}
\newcommand{\stella}{\prodstar\displaylimits}

\title{\bf Landau-Toeplitz theorems for slice regular functions over quaternions}

\author{Graziano Gentili\footnote{Partially supported by GNSAGA of the INdAM and by PRIN ``Propriet\`a geometriche delle variet\`a reali e complesse'' of the MIUR.}  \\ 
\normalsize Dipartimento di Matematica ``U. Dini'', Universit\`a di Firenze \\ 
\normalsize Viale Morgagni 67/A, 50134 Firenze, Italy,  gentili@math.unifi.it \\
\and Giulia Sarfatti$^*$\footnote{Partially supported by PRIN ``Geometria Differenziale e Analisi Globale'' of the MIUR.} \\ 
\normalsize Dipartimento di Matematica ``U. Dini'', Universit\`a di Firenze \\ 
\normalsize Viale Morgagni 67/A, 50134 Firenze, Italy,  sarfatti@math.unifi.it}

\date{}
\begin{document}

\maketitle

\begin{abstract} The theory of slice regular functions of a quaternionic variable, as presented in \cite{libroGSS},
extends the notion of holomorphic function to the quaternionic setting. This theory, already rich of results, is sometimes surprisingly 
different from the theory of holomorphic functions of a complex variable. However, several fundamental results in the two environments are similar, 
even if their proofs for the case of quaternions need new technical tools. 

In this paper we prove the Landau-Toeplitz Theorem for slice regular functions, in a formulation that involves an 
appropriate notion of \emph{regular $2$-diameter}. We then show that the Landau-Toeplitz inequalities hold in the case of the 
regular $n$-diameter, for all $n\geq 2$. Finally, a $3$-diameter version of the Landau-Toeplitz Theorem is proved using the notion of {\em slice $3$-diameter}. 
\end{abstract}



\section{Introduction}

The Schwarz Lemma, in its different flavors, is the basis of a chapter of fundamental importance in the geometric theory of 
holomorphic functions of one and several complex variables. Its classic formulation in one variable is the following:

\begin{teo}[Schwarz Lemma]\label{schwarz} Let $\D=\{z\in \mathbb{C} : |z|<1\}$ be the open unit disc of $\mathbb{C}$ centered at the origin, 
and let $f:\D\rightarrow \D$ be a holomorphic function such that $f(0)=0$. Then
\begin{equation}\label{i}
|f(z)| \leq |z|
\end{equation}
for all $z \in \D$, and
\begin{equation}\label{ii}
|f'(0)|\leq 1.
\end{equation}
Equality holds in \eqref{i} for some $z \in \D \setminus \{0\}$, or in \eqref{ii}, if and only if there exists $u \in \mathbb{C}$, 
with $|u|=1$, such that $f(z)=uz$ for all $z \in \D$.
\end{teo}
\noindent The Schwarz Lemma, and its extension due to Pick, lead in a natural way to the construction of the Poincar\'e metric, 
which plays a key role in the study of the hyperbolic geometry of complex domains and manifolds. In 1907, the same year of the first formulation of 
the Schwarz Lemma, the Landau-Toeplitz Theorem was proven, \cite{L-T}. This less known, but quite interesting result, 
concerns the study of the possible ``shapes'' of the image of the unit disc under a holomorphic function and it is formulated in terms 
of the diameter of the image set. 

\begin{teo}[Landau-Toeplitz]
Let $f$ be holomorphic in  $\D$ and such that the diameter $\Diam f(\D)$ of $f(\D)$ equals  $2$. Then
\begin{equation}\label{LTI1}
\Diam f(r\D) \leq 2r 
\end{equation}
for all $r \in (0,1)$,
and
\begin{equation}\label{LTI2}
|f'(0)|\leq 1.
\end{equation}
Moreover equality holds in \eqref{LTI1} for some $r \in (0,1)$, or in \eqref{LTI2}, if and only if $f$ is of the form $f(z)=a +bz$ 
with $a,b \in \mathbb{C}$ and $|b|=1$.
\end{teo}

\noindent This result can be interpreted as a generalization of the classical Schwarz Lemma, in which the diameter of the image set 
takes over the role of the maximum modulus of the function. Indeed, there exist infinite subsets of the plane that have constant diameter and 
are different from a disc; the  Reuleaux Polygons are a well known example of such sets, \cite{gardner, L-R}.

The recent definition of slice regularity for quaternionic functions of one quaternionic variable, inspired by Cullen \cite{Cullen} 
and developed in \cite{G.S.} and in \cite{GSAdvances} by Gentili and Struppa, identifies a large class of functions, that includes natural 
quaternionic power series and polynomials. The study of a geometric theory for this class of functions has by now produced several interesting 
results, sometimes analogous to those valid for holomorphic functions; the Schwarz Lemma is among these results, \cite{GSAdvances}, together with the Bohr Theorem and the Bloch-Landau Theorem, \cite{BohrDGS},\cite{BlochDGS},\cite{sarfatti}.

Fairly new developments in the theory of holomorphic functions of one complex variable  include the analogue of the Schwarz Lemma for 
meromorphic functions, and open new fascinating perspectives for future researche. In this setting the paper \cite{Solynin} by Solynin  
recalls into the scenery the approach of Landau and Toeplitz, and its modern reinterpretation and generalization due to Burckel, Marshall, Minda, 
Poggi-Corradini and Ransford, \cite{Poggi}. 

In our paper we first prove an analogue of the Landau-Toeplitz Theorem for slice regular functions. 
To this purpose we need to introduce a new tool to ``measure'' the image of the open unit ball $\B$ of the space of quaternions $\HH$ 
through a slice regular function: the \emph{regular diameter}.
\begin{defn}\label{Id2}
Let $f$ be a slice regular function on $\B=\{q\in \mathbb{H} : |q|<1\}$ and let $f(q)=\sum_{n \geq 0}q^na_n$ be its power series expansion. 
For $r \in (0,1)$, we define the {\em regular diameter} of the image of $r\B$ under $f$ as
$$\tilde{d}_2(f(r\B))=\max_{ u,v \in \overline{\B} }\max_{|q| \leq r}|f_u(q)-f_v(q)|,$$ 
where 
$$f_u(q)=\sum_{n\geq 0}q^nu^na_n, \hskip 1cm f_v(q)=\sum_{n\geq 0}q^nv^na_n.$$
Moreover, we define the {\em regular diameter} of the image of $\B$ under $f$ as
$$\tilde{d}_2(f(\B))=\lim_{r \to 1^-}\tilde{d}_2(f(r\B)).$$ 
\end{defn}

\noindent The introduction of this new geometric quantity is necessary because of the peculiarities of the quaternionic environment, 
and in particular since 
composition of slice regular functions is not slice regular in general. The regular diameter can play the role of the diameter, in fact the former 
is finite if and only if the latter is finite. The regular diameter hence appears in the statement of the announced result.
 
\begin{teo}[Landau-Toeplitz for slice regular functions]\label{007}
Let $f$ be a slice regular function on $\B$ such that $\tilde{d}_2(f(\B))=2$ and let $\partial_c f(0)$ be its \emph{slice derivative} in $0$. Then 
\begin{equation}\label{ILTQ2}
\tilde{d}_2(f(r\B)) \leq 2r \quad \text{for all} \quad r \in (0,1)
\end{equation}
and
\begin{equation}\label{ILTQ1}
|\partial_c f(0)|\leq 1.
\end{equation}
Moreover, equality holds in \eqref{ILTQ2} for some $r \in (0,1)$, or in \eqref{ILTQ1}, 
 if and only if $f$ is an affine function, i.e. $f(q)=a +qb$, with $a,b \in \mathbb{H}$ and $|b|=1$.
\end{teo}
\noindent As in the complex setting, Theorem \ref{007} can be interpreted as a generalization of the Schwarz Lemma.

The new version of the Landau-Toeplitz Theorem proposed in \cite{Poggi} concerns holomorphic functions whose image is measured with a 
notion of diameter more general than the classic one, the $n$-diameter. 
In the quaternionic setting, the analogue of this geometric quantity is defined as
\begin{defn}\label{n-diametro0}
Let $E \subset \HH$. For every $n \in \mathbb{N}$, $n \geq 2$, the {\em $n$-diameter} of $E$ is defined as
$$d_n(E)= \sup_{w_1,...,w_n \in E} \Big( \prod_{1\leq j<k \leq n} |w_k -w_j|\Big)^{\frac{2}{n(n-1)}}.$$
\end{defn}
\noindent Retracing the approach used in the complex setting, we are able to obtain only the generalization of the first part of the 
statement of the Landau-Toeplitz Theorem for the $n$-diameter. As in the case $n=2$, we need a notion of regular $n$-diameter $\tilde d_n(f(\B))$ 
for the image  of $\B$ through a slice regular function $f$. This notion is a 
generalization of Definition \ref{Id2}, modeled on Definition \ref{n-diametro0} and given in terms of the $*$-product between slice regular 
functions (see Section \ref{preliminari}).
%
For all $n\geq 2$, the regular $n$-diameter turns out to be finite when the $n$-diameter is finite. For this reason, even if it may appear 
awkward, it makes sense to use the regular $n$-diameter in the following statement.

\begin{teo}
Let  $f$ be a slice regular function on  $\B$ such that $\tilde{d}_n(f(\B))=d_n(\B)$. Then
\begin{equation*}
\tilde{d}_n(f(r\B))\leq d_n(r\B) \quad \text{for all} \quad r \in (0,1)
\end{equation*}
and
\begin{equation*}
|\de f(0)|\leq 1.
\end{equation*}
\end{teo}

\noindent 

Since the $3$-diameter of a $4$-dimensional subset of $\HH$ is attained on a (specific) bidimensional section, we are encouraged to introduce  an appropriate notion  $\hat{d}_3 f(\B)$ of  \emph{slice $3$-diameter} for $f(\B)$, inspired by the power series expansion of the regular $3$-diameter.  This leads to the following complete result.
\begin{teo}[Landau-Toeplitz Theorem for the slice $3$-diameter]
Let $f$ be a slice regular function on $\B$ such that $\hat{d}_3 (f(\B))=d_3(\B)$. Then
\begin{equation}\label{d31I}
\hat{d}_3 (f(r\B))\leq d_3(r\B) \quad \text{for every} \quad r\in (0,1)
\end{equation}
and
\begin{equation}\label{d32I}
|\partial_c f(0)|\leq 1.
\end{equation}
Moreover equality holds in \eqref{d31I}, fore some $r \in (0,1)$, or in \eqref{d32I}, if and only if $f$ is an affine function, $f(q)=a + qb$ with 
$a,b \in \HH$, and $|b|=1$.
\end{teo}
We point out that all the extensions of the Landau-Toeplitz results presented in this paper generalize the Schwarz Lemma to a much larger 
class of image sets. In fact, for all $n\geq 2$, there exist infinitely many subsets of the space $\HH$, which have fixed $n$-diameter, 
do not coincide with a $4$-ball and neither contain, nor are contained, in the $4$-ball. The $4$-bodies of constant width are examples of 
such subsets, presented for instance in \cite{gardner, L-R}.

\section{Preliminaries}\label{preliminari}
Let $\HH$ be the skew field of quaternions, obtained by endowing $\mathbb{R}^4$ with the multiplication operation 
defined on the standard basis $\{1,i,j,k\}$ 
by $i^2=j^2=k^2=-1$ and $ij=k$, \
and then extended by distributivity to all quaternions $q= x_0 +x_1i+x_2j+x_3k$. 
For every $q \in \HH$, we define the {\em real} and {\em imaginary} part of $q$ as $\RRe(q)=x_0$ and $\IIm(q)=x_1i +x_2j +x_3k$, 
its {\em conjugate} as $\bar{q}=\RRe(q)-\IIm(q)$ and its {\em modulus} by $|q|^2=q\bar{q}$. The multiplicative inverse of 
each $q\neq0$ is computed as $q^{-1}=\bar{q}/|q|^2$.
Let $\mathbb{S}$ be the unit $2$-sphere of purely imaginary quaternions, $\mathbb{S}=\{q\in \HH  \ |\ q^2=-1 \}.$
Then, for any $I\in\s$, we will denote by $L_I$ the complex plane $\mathbb{R}+\mathbb{R}I$, 
and, if $\Omega \subset \HH$, we further set $\Omega_I=\Omega \cap L_I$.
Notice that to every $q \in \HH \setminus \mathbb{R}$, we can associate a unique element in $\mathbb{S}$ by the map
$q \mapsto \IIm(q)/|\IIm(q)|=I_q.$
Therefore, for any $q \in \HH \setminus \mathbb{R}$, there exist and are unique, $x, y \in \mathbb{R}$, with $y>0$ and $I_q \in \mathbb{S}$, 
such that $q=x+yI_q$. If $q$ is real, then $I_q$ can be any element of $\mathbb{S}$.

The preliminary results stated in this section will be given for slice regular functions defined on open balls of type 
$B=B(0,R)=\{q \in \HH \, |\, |q|<R \}$. We point out that, in most cases, these results hold, with appropriate changes, for a more general class of domains, introduced in \cite{ext}. 
Let us now recall the definition of slice regularity.
\begin{defn}
A function $f: B=B(0,R) \rightarrow \HH$ is said to be {\em slice regular} if, for every $I\in \mathbb{S}$,  
its restriction $f_I$ to $B_I$, has continuous partial derivatives and satisfies
$$\overline{\partial}_I f(x+yI)=\frac{1}{2}\Big(\frac{\partial}{\partial x}  +I\frac{\partial}{\partial y}\Big)f_I(x+yI)=0$$
for every $x+yI \in B_I$.
\end{defn}
\noindent In the sequel we may refer to the vanishing of $\overline{\partial}_If$ saying that the restriction $f_I$ is holomorphic on $B_I$.
In what follows, for the sake of shortness, we will omit the prefix slice when referring to slice regular functions.
\noindent A notion of derivative, called {\em slice (or Cullen) derivative},  can be given for regular functions by
\begin{equation*}
\de f(x+yI)= \frac{\partial}{\partial x} f(x+yI),
\end{equation*}
\noindent for $x+yI\in B$. This definition is well posed because it is applied only to regular functions and moreover that slice regularity is preserved by slice differentiation.
A basic result connects slice regularity and classical holomorphy, \cite{GSAdvances}:
\begin{lem}[Splitting Lemma]\label{split}
If $f$ is a regular function on $B=B(0,R)$, then for every $I \in \mathbb{S}$ 
and for every $J \in \mathbb{S}$, $J$ orthogonal to $I$, there exist two holomorphic functions $F,G:B_I \rightarrow L_I$, such that 
for every $z=x+yI \in B_I$, it results
$$f_I(z)=F(z)+G(z)J.$$
\end{lem}
\noindent As proven in \cite{GSAdvances}, 

\begin{teo}
A function $f$ is regular on $B=B(0,R)$ if and only if $f$ has a power series expansion
$$f(q)=\sum_{n \geq 0}q^na_n \quad\text{with} \quad a_n=\frac{1}{n!}\frac{\partial^n f}{\partial x^n}(0)$$
converging absolutely and uniformly on compact sets in $B(0, R)$.
\end{teo}  
\noindent In the sequel we will also need the Identity Principle for regular functions, stated here in its ``weak'' version, \cite{GSAdvances}.
\begin{teo}[Identity Principle]\label{Id}
Let $f:B=B(0,R) \rightarrow \HH$ be a regular function. Denote by $Z_f$ the zero set of $f$, $Z_f=\{ q \in B | \, f(q)=0 \}$. 
If there exists $I \in \mathbb{S}$ such that $B_I \cap  Z_f$ has an accumulation point in $B_I$, then
$f$ vanishes identically on $B$.
\end{teo}
\noindent A basic result, that will be useful in the sequel, is the following (see \cite{ext}).
\begin{teo}[Representation Formula]\label{RF}
Let $f$ be a regular function on $B=B(0,R)$ and let $J\in \mathbb{S}$. Then, for all $x+yI \in B$, the following equality holds
\begin{equation*}
f(x+yI)=\frac{1}{2}\big[ f(x+yJ)+f(x-yJ)\big]+I\frac{1}{2}\big[J\big[f(x-yJ)-f(x+yJ) \big]\big].
\end{equation*}
\end{teo}

\noindent The product of two regular functions is not, in general, regular. To guarantee the regularity we need to introduce the following 
multiplication operation, denoted by the character $*$.
\begin{defn} 
Let $f(q)=\sum_{n \geq 0}q^na_n$ and $g(q)=\sum_{n\geq 0}q^nb_n$ be regular functions on $B=B(0,R)$. The {\em $*$-product} of $f$ and $g$ 
is the regular function $f*g: B \rightarrow \HH$ defined by
$$f*g(q)=\sum_{n \geq 0}q^n\sum_{k=0}^na_k b_{n-k}.$$
\end{defn}
\noindent Notice that the $*$-product is associative and is not, in general, commutative. 
The following result clarifies the relation between the $*$-product and the pointwise product of regular functions, \cite{libroGSS}.
\begin{pro}\label{trasf}
Let $f(q)=\sum_{n\geq 0}q^na_n$ and $g(q)=\sum_{n\geq 0}q^nb_n$ be regular functions on $B=B(0,R)$. Then
\begin{equation*}\label{prodstar}
f*g(q)= \left\{ \begin{array}{ll}
 f(q)g(f(q)^{-1}qf(q)) & \text{if} \quad f(q)\neq 0\\
0 & \text{if} \quad f(q)=0
\end{array}
\right.
\end{equation*}
\end{pro}
\noindent 
Notice that if $q=x+yI$ (and if $f(q)\neq 0$), then $f(q)^{-1}qf(q)$ has the same modulus and same real part as $q$, hence $T_f(q)=f(q)^{-1}qf(q)$ 
lies in $x+y\mathbb{S}$, the same $2$-sphere as $q$. We have that a zero $x_0+y_0I$ of the function $g$ is not necessarily a zero of $f*g$, 
but one element on the same sphere $x_0+y_0\mathbb{S}$ does.

To conclude this preliminary section we recall a result that is basic for our purposes 
(see \cite{GSAdvances}).
\begin{teo}[Maximum Modulus Principle]\label{PMM}
Let $f: B \rightarrow \HH$ be a regular function. If there exists $I\in\s$ such that the restriction $|f_I|$ has a local maximum in $B_I$, then $f$ is constant in $B$.
In particular, if $|f|$ has a local maximum in $B$, then $f$ is constant in $B$. 
\end{teo}

\section{The Landau-Toeplitz Theorem for regular functions}\label{sec2.1}

In this section we will prove the analogue of the celebrated Landau-Toeplitz
Theorem for holomorphic functions, \cite{Pog2, L-T}, in the new setting of (quaternionic)
regular functions. To reach the aim, we will need a few steps.

Denote by $\langle \ ,\  \rangle$ the scalar product of $ \mathbb{R}^4$, and by $\times$ the vector product of $\rr^3$. Recall that for all purely 
imaginary quaternions $u,v$ the customary equality $uv=-\langle u, v\rangle +u\times v$ holds. If $w=x+yL\in \HH$, then for all $I\in \s$, $\langle w, I \rangle = \langle yL, I\rangle = -\RRe (yLI)= -\RRe (wI)$.

\begin{defn}\label{comp}
Let $I \in \mathbb{S}$. For any  $w \in \HH$ we define the imaginary component of $w$ along $I$ as
$\IIm_{I}(w)=\langle w, I \rangle=-\RRe (wI)$.
\end{defn}  

\noindent We are  now ready to prove a first preliminary result.

\begin{pro}\label{Imq2}
Let $w \in B=B(0,R)$, $0<|w|=r<R$, and let $g$ be a holomorphic function on $ B\cap L_{I_w}$. If  
\begin{equation}\label{ipowolff}
g(w)=w \quad \text{and} \quad r=\max_{z \in r\overline {\B}_{I_w}} |g(z)|, 
\end{equation}
then $\IIm_{I_w} ( \partial_c g(w))=0$.
\end{pro}

\begin{proof}
To simplify the notation, let us set $I=I_w$.
Define $\varphi: \mathbb{R} \rightarrow \mathbb{R}$ as the function 
$\varphi(\theta)=|g(we^{I\theta})|^2.$
The Splitting Lemma implies that for every $J \in \mathbb{S}$ orthogonal to $I$ there exist $F,G: B_I \rightarrow L_I$, holomorphic 
functions, such that 
$
 g(z)=F(z)+G(z)J$ for every $z \in B_I.
$
A direct computation shows that
$
\varphi(\theta)=  F(we^{I\theta})\overline{F(we^{I\theta})} + G(we^{I\theta})\overline{G(we^{I\theta})}.
$\,
Hence
\begin{equation*} 
\begin{aligned}
\varphi'(\theta)&=-2 \IIm_I \big(we^{I\theta}\big(F'(we^{I\theta})\overline{F(we^{I\theta})}+
G'(we^{I\theta})\overline{G(we^{I\theta})}\big)\big),
\end{aligned}
\end{equation*}
where $F'$ and $G'$ denote, respectively, the complex derivatives of $F$ and
$G$ in $B_I$.
Since, by hypothesis, $\theta=0$ is a maximum for $\varphi$, it turns 
out 
that
\begin{equation}\label{fi}
0= \varphi'(0)=-2\IIm_I \big(w\big(F'(w)\overline{F(w)}+ G'(w)\overline{G(w)}\big)\big).
\end{equation}
Moreover $w=g(w)=F(w)+G(w)J ,$ 
which implies $F(w)=w$ and $G(w)=0.$
Putting this values in equation \eqref{fi} we have
$ 0=-2\IIm_I (w F'(w)\bar{w})=-2|w|^2 \IIm_I(F'(w) )$
which yields $\IIm_I(F'(w))=0.$
Finally, recalling the definition of the slice derivative, and Definition \ref{comp},  we get 
$$\IIm_{I} (\partial_c g(w))=\IIm_{I}\big(F'(w)+G'(w)J\big)=\IIm_{I}(F'(w))=0.$$
\end{proof}

\begin{oss}
The previous result can be interpreted as a consequence of the 
Julia-Wolff-Carath\'eodory Theorem (see for instance \cite{abate, Burckel}). In fact hypotheses \eqref{ipowolff} yield that
$g:r\B_{I}\to r\B$ and that $w$ is a boundary fixed point for the restriction of $g$ to $r\B_I$. 
Hence, if we split  the function $g$ as $g(z)=F(z)+G(z)J$, for $z\in r\B_{I}$, 
we have that  $w$ is a Wolff point for $F: r\B_I \to r\B_I$.
\end{oss}

The proof of the classical Landau-Toeplitz Theorem in the 
setting of holomorphic maps, \cite{Poggi},  relies upon the analogue of Proposition
\ref{Imq2}, which  is not sufficient for our purposes in the
quaternionic environment. In fact we need the following

\begin{pro}\label{Imq02}
Let $g:\B \rightarrow \HH$ be a regular function such that $
\IIm_{I_q} (g(q))=0
$ for every $q \in \B$. Then $g$ is a real constant function. 
\end{pro}

\begin{proof}
Let $g(q)=\sum_{n \geq 0}q^n a_n$ on $\B$. For any $I\in \s$ we split the coefficient $a_n$ as $b_n+c_nJ$ with $b_n, c_n \in L_I$ and $J\in \s$ orthogonal to $I$. By hypothesis we have
\[
0=\IIm_I(g(z))=\IIm_I\Big(\sum_{n \geq 0}z^n (b_n+c_nJ)\Big)=\IIm_I\Big(\sum_{n \geq 0}z^n b_n\Big)
\]
for all $z\in \B_I$. As a consequence of the Open Mapping Theorem the holomorphic map $\sum_{n \geq 0}z^n b_n$ is constant, i.e., $b_n=0$ for all $n>0$. Therefore the component of each $a_n$ along $L_I$ vanishes for all $n>0$. Since $I\in \s$ is arbitrary, this implies $a_n=0$ for all $n>0$. The hypothesis yields that $a_0\in \rr$.

\end{proof}

A basic notion used to state the classical 
Landau-Toeplitz Theorem is the diameter of the images of holomorphic
functions. In the new quaternionic
setting, due to the fact that a composition of regular functions is not regular
in general, the definition of a ``regular'' diameter for the images of
regular functions requires a peculiar approach.

\begin{defn}
    
Let $f:\B \rightarrow \HH$ be a regular function, $f(q)=\sum_{n \geq 0}q^n
a_n$, 
and let $u\in \mathbb{H}$. We define the {\em regular composition} 
of $f$ with the function $q\mapsto qu$ as
$$
f_u(q)=\sum_{n \geq 0}(qu)^{*n} a_n=\sum_{n \geq 0}q^n u^n a_n.
$$
\end{defn}
\noindent Notice that, if $|u|=1$,  the radius of convergence of the series
expansion for $f_u$ is the same as the one for $f$.
Moreover, if $u$ and $q_0$ lie in the same plane $L_I$,  then $u$ and $q_0$ commute
 and hence $f_u(q_0)=f(q_0u)$. In particular, if $u \in \mathbb{R}$,
 then $f_u(q)=f(qu)$ 
for every $q$.
We now have  all the necessary tools to give the following
\begin{defn}
Let $f:\B \rightarrow \HH$ be a regular function. For $r \in (0,1)$, we define the {\em regular diameter} of the image of $r\B$ under $f$ as 
$$\tilde{d}_2(f(r\B))=\max_{ u,v \in \overline{\B} }\max_{|q| \leq r}|f_u(q)-f_v(q)|.$$
Moreover we define the {\em regular diameter} of the image of   $\B$ under $f$ as
\begin{equation}\label{diametror}
\tilde{d}_2(f(\B))=\lim_{r \to 1^{-}}\tilde{d}_2(f(r\B)).
\end{equation}
\end{defn}
\begin{oss}\label{increasing} By the Maximum Modulus Principle for regular functions, $\tilde{d}_2(f(r\B))$ is an increasing 
    function of $r$, 
and hence limit \eqref{diametror} always exists. 
So $\tilde{d}_2(f(\B))$ is well defined.
\end{oss}

\noindent Let $E$ be a subset of $\HH$. We will denote by $\, \Diam E=\displaystyle{\sup_{q,w\in E}|q-w|},\, $
the classical diameter of $E$.

\begin{pro}\label{2d2}
Let $f$ be a regular function on $\B$. Then the following inequalities hold
\[\Diam (f(\B))\le\tilde{d}_2(f(\B))\leq 2\Diam(f(\B)).\]
\end{pro}
\begin{proof}
In order to prove the first inequality, let $r\in(0,1)$ and consider $q,w \in r\overline{\B}$. We want to bound the quantity $|f(q)-f(w)|$. 
Suppose without loss of generality that $|w|\geq |q|$ and moreover that $w\neq 0$. We have
\begin{equation}\label{sopra}
\begin{aligned}
|f(q)-f(w)|&=\left|f\left(q \frac{|w|}{|w|}\right)- f\left(w\frac{|w|}{|w|}\right)\right|=\left|f_{\frac{q}{|w|}}\left(|w|\right)- f_{\frac{w}{|w|}}\left(|w|\right)\right|.
\end{aligned}
\end{equation}
Where the last equality is due to the fact that $|w|$, being real, commutes with both $q/|w|$ and $w/|w|$. 
Moreover, since $q/|w|\in \overline{\B}$ and $w/|w|\in \p \B$ equation \eqref{sopra} yields
\[|f(q)-f(w)|\le\max_{u,v\in \overline{\B}}|f_u(|w|)-f_v(|w|)|\le\max_{u,v\in \overline{\B}}\max_{|q|\le r}|f_u(q)-f_v(q)|=\tilde{d}_2(f(r\B)).\]
This implies that  $\, \Diam(f(r\overline{\B}))\le\tilde{d}_2(f(r\B)).$
Since the previous inequality holds for any $r\in(0,1)$, we obtain that
\[\Diam(f(\B))=\lim_{r\to 1^-}\Diam(f(r\overline{\B}))\le\lim_{r\to 1^-}\tilde{d}_2(f(r\B))=\tilde{d}_2(f(\B)).\]

To show the missing inequality, let $u,v \in \overline{\B}$, $r\in (0,1)$, and let $J,K$ be elements of $\mathbb{S}$ 
such that $u \in L_J$ and $v \in L_K$. Using the Representation Formula \ref{RF}, and taking into account that $u$ and $x+yJ$ commute as well as $v$ and $x+yK$, 
we get that for all $q=x+yI \in r\overline{\B}$  

\begin{equation}\label{finito}
\begin{aligned}
&|f_u(q)-f_v(q)|=\frac{1}{2}| ( f((x+yJ)u)-(f(x+yK)v))+( f((x-yJ)u)-f((x-yK)v)) \\
&+IJ(f((x-yJ)u)-f((x+yJ)u))-I K(f((x-yK)v)-f((x+yK)v)) | \\
&\leq \frac{1}{2}| f((x+yJ)u)-(f(x+yK)v)| +\frac{1}{2}| f((x-yJ)u)-f((x-yK)v)|\\
&+\frac{1}{2}|f((x-yJ)u)-f((x+yJ)u)|+\frac{1}{2}| f((x-yK)v)-f((x+yK)v)|\\
&\leq 2\Diam f(r\B).
\end{aligned}
\end{equation}
Since inequality \eqref{finito} holds for every $u,v \in \overline{\B}$ and for every $q \in r\overline{\B}$, we get 
\begin{equation}\label{limitazione}
\tilde{d}_2(f(r\B))=\max_{ u,v \in \overline{\B} }\max_{|q| \leq r}|f_u(q)-f_v(q)|\leq 2 \Diam f(r\B).
\end{equation}
Moreover, since inequality \eqref{limitazione} holds for every $r\in (0,1)$, we obtain 
$\tilde{d}_2(f(\B))\le2\Diam(f(\B)).$
\end{proof}

Notice that if $f$ is an affine function, say $f(q)=a +qb$, then
$\tilde{d}_2(f(r\B))=|b|\Diam (r\B)=|b|r\Diam(\B)$ for every $ r \in
(0,1).$
In particular, if $f$ is constant, then $\tilde{d}_2(f(r\B))=0$.
Moreover, the regular diameter $\tilde{d}_2(f(r\B))$ is invariant under
translations; in fact, if $g(q)=f(q)-f(0)$, then 
$\tilde{d}_2(g(r\B))=\tilde{d}_2(f(r\B))$ for every $r \in (0,1)$.

\begin{teo}[Landau-Toeplitz for regular functions]\label{f0q}
Let $f:\B\rightarrow \HH$ be a regular function such that $\tilde{d}_2(f(\B))=\Diam \B=2$. Then 
\begin{equation}\label{LTQ2}
\tilde{d}_2(f(r\B)) \leq 2r \quad \text{for every} \quad r \in (0,1)
\end{equation}
and
\begin{equation}\label{LTQ1}
|\partial_c f(0)|\leq 1.
\end{equation}
Equality holds in \eqref{LTQ2} for some $r \in (0,1)$, or in \eqref{LTQ1}, 
 if and only if $f$ is an affine function, $f(q)=a +qb$, with $a,b \in \mathbb{H}$ and $|b|=1$.
\end{teo}

\begin{proof}
To prove the first inequality, take $u,v\in \overline{\B}$, and consider the auxiliary function
\[g_{u,v}(q)= \frac{1}{2}q^{-1}(f_u(q)-f_v(q)).\]
This function is regular on $\B$. Indeed, if the power series expansion of $f$ 
in $\B$ is 
$
\sum_{n \geq 0}q^n a_n,
$
then it turns out that
\begin{equation*}\label{h}
\begin{aligned}
g_{u,v}(q)&=\frac{1}{2}q^{-1}\Big( \sum_{n\geq 0}q^n u^n a_n - \sum_{n\geq 0}q^n v^n a_n\Big)=\frac{1}{2}\sum_{n\geq 0}q^n(u^{n+1}-v^{n+1}) a_{n+1}.
\end{aligned}
\end{equation*}
From this expression of $g_{u,v}$ we can recover its value at $q=0$
\begin{equation}\label{h0}
g_{u,v}(0)= \frac{1}{2}(u-v)a_1=\frac{1}{2}(u-v) \partial_c f(0).
\end{equation}
Since $g_{u,v}$ is a regular function, using the Maximum Modulus Principle, we get that 
\[r \mapsto \max_{u,v\in \overline{\B}}\max_{|q|\le r}|g_{u,v}(q)|\]
is increasing on $(0,1)$.
Moreover,
the regularity of the function $q\mapsto f_u(q)-f_v(q)$ yields that, 
for any fixed $r\in (0, 1)$, we can write 
\begin{equation*}
\begin{aligned}
&\max_{|q|\leq r}|g_{u,v}(q)|=\max_{|q|\leq r} \frac{|f_u(q)-f_v(q)|}{2|q|}=\frac{\max_{|q|\leq
r}|f_u(q)-f_v(q)|}{2r},
\end{aligned}
\end{equation*}
which leads to
\begin{equation}\label{Dr}
    \begin{aligned}
\frac{\tilde{d}_2(f(r\B))}{2r}&= \frac{\max_{u,v \in \overline{\B}} \max_{|q|\leq
r}|f_u(q)-f_v(q)|}{2r}=\max_{u,v \in \overline{\B}}\max_{|q|\leq
r}|g_{u,v}(q)|.
\end{aligned}
\end{equation}
Therefore $\tilde{d}_2(f(r\B))/2r$ is an increasing function of $r$ and hence it is always less than or equal to 
the limit
\[ \lim_{r \to 1^{-}} \frac{\tilde{d}_2(f(r\B))}{2r}=\frac{\tilde{d}_2(f(\B))}{2}= 1.\]
This means that 
\begin{equation}\label{Drq}
\tilde{d}_2(f(r\B)) \leq 2r \quad \text{for every} \quad r \in (0,1),
\end{equation}
proving hence inequality \eqref{LTQ2} of the statement.
To prove the second inequality, Consider the odd part of $f$, 
\[f_{odd}(q)=\frac{f(q)-f(-q)}{2}.\]
Notice that $f_{odd}$ satisfies the hypotheses of the Schwarz Lemma for regular functions (see \cite{GSAdvances}). Indeed
$f_{odd}$ is a regular function on $\B$, $f_{odd}(0)=0$, and
\[|f_{odd}(q)|=\frac{|f(q)-f(-q)|}{2}\leq \frac{ \tilde{d}_2(f(\B))}{2}= 1\]
for every $q \in \B$.
Hence
\begin{equation}\label{deriv}
1\geq |\partial_c f_{odd}(0)|=\frac{|\partial_c f(q) - \partial_c (f(-q))|}{2}_{\big|_{q=0}}=\frac{|\partial_c f(q) + \partial_c f(-q)|}{2}_{\big|_{q=0}}=|\partial_c f(0)|.
\end{equation}

We will now prove the last part of the statement, covering the case of equality. To begin with, notice that if $f(q)=a+ qb$ with $a,b \in \mathbb{H}$ and $|b|=1$, then equality 
holds in both \eqref{LTQ2} and \eqref{LTQ1}.\\
Conversely, suppose that equality holds in \eqref{LTQ1}, namely that $|\partial_c f(0)|=1$.
In this case we have $|\partial_c f_{odd}(0)|=1$ and therefore, by the Schwarz Lemma (see \cite{GSAdvances}),  
\begin{equation}\label{fd}
f_{odd}(q)= q \partial_c f(0).
\end{equation}
We want to show that in this case $ \tilde{d}_2(f(r\B))=2r$ for every $r \in (0,1)$. In fact,
from \eqref{h0} and \eqref{Dr} it follows 
\[\frac{\tilde{d}_2(f(r\B))}{2r} \geq \max_{u,v \in \overline{\B}} |g_{u,v}(0)|= \max_{u,v \in \overline{\B}} \frac{1}{2}|(u-v) \partial_c f(0)|=1 \quad \text{for every} \quad r \in (0,1).\] 
Comparing the last inequality with \eqref{Drq} we get 
\begin{equation}\label{2r}
\tilde{d}_2(f(r\B))=2r \quad \text{for every} \quad r \in (0,1).
\end{equation}
We now introduce a new auxiliary function. Take $w \in \B$, with $0<|w|=r<1$ and set 
\[h_w(q)= \frac{1}{2}(f(q)-f(-w))\partial_c f(0)^{-1}.\]
The function $h_w$ is  regular on $\B$ and fixes $w$; indeed 
\[h_w(w)=\frac{1}{2}(f(w)-f(-w))\partial_c f(0)^{-1}=f_{odd}(w)\partial_c f(0)^{-1}=w \]
where the last equality is due to \eqref{fd}.
We need now to restrict our attention to what happens in $L_{I_{w}}$. 
By the Maximum Modulus Principle \ref{PMM},
 we are able to find $z_0 \in L_{I_w}$, $|z_0|=r$, such that for $z\in L_{I_{w}}$
 $$\max_{|z|\leq r}|h_w(z)|=\frac{1}{2}\max_{|z|\leq r}|f(z)-f(-w)|=\frac{1}{2}|f(z_0)-f(-w)|.$$
 Let $\hat{u} \in L_{I_w}$ with $|\hat{u}|=1$ be such that $-w=z_0\hat{u}$.
 Then, again for  $z\in L_{I_{w}}$, due to the fact that $z_0$ and $\hat{u}$ commute
 
 $$\max_{|z|\leq r}|h_w(z)|=\frac{1}{2}|f(z_0)-f(z_0\hat{u})|=\frac{1}{2}|f(z_0)-f_{\hat{u}}(z_0)|\leq \frac{1}{2}\max_{u,v \in \overline{\B}}\max_{|z|\leq r}|f_u(z)-f_v(z)|.$$
 Recalling \eqref{2r}, 
 for $z\in L_{I_{w}}$ and $q\in \mathbb{H}$ we obtain
 $$\max_{|z|\leq r}|h_w(z)|\leq \frac{1}{2} \max_{u,v \in \overline{\B}}\max_{|q|\leq r}|f_u(q)-f_v(q)|=\frac{1}{2}\tilde{d}_2(f(r\B))= r =|h_w(w)|.$$

\noindent The function $h_w$ then satisfies the hypotheses of Proposition \ref{Imq2}, and hence 
\[ 
0=\IIm_{I_w} \left(\partial_c h_w(q)_{|_{q=w}}\right)=\IIm_{I_w} \left(\frac{1}{2} \partial_c f(w)\partial_c f(0)^{-1}\right).
\]
Now recall that $w$ is an arbitrary element of $\B \setminus \{0\}$. By continuity, we get that the 
function $w\mapsto \frac{1}{2}\partial_c f(w)\partial_c f(0)^{-1}$, regular on $\B$, satisfies the hypotheses of Proposition \ref{Imq02}. 
Consequently $\frac{1}{2}\partial_c f(w)\partial_c f(0)^{-1}$ is a real constant function and hence $\partial_c f(w)$ is constant as well. 
Therefore $f$ has the required form $f(q)= f(0) + q\partial_c f(0)$.

We will show now how equality in \eqref{LTQ2} for some $s\in (0, 1)$
implies  equality in \eqref{LTQ1}. This and the preceding step will conclude the
proof. 
Suppose that there exists $s \in (0,1)$ such that  $\tilde{d}_2(f(s\B))/2s=1$. By inequality \eqref{Drq} and since $\tilde{d}_2(f(r\B))/2r$ 
is increasing in $r$, we have
\[\frac{\tilde{d}_2(f(r\B))}{2r}=1 \quad \text{for every} \quad r \in [s,1).\]
Let us prove that this equality holds for  all $r \in (0,1).$
To do this, let $\hat{u}, \hat{v} \in \overline{\B}$, be such that 
\begin{equation*}\label{ug1}
\frac{\tilde{d}_2(f(s\B))}{2s}=\max_{u,v \in \overline{\B}} \max _{|q| \leq s}|g_{u,v}(q)|=\max_{|q|\leq s}|g_{\hat{u},\hat{v}}(q)|
\end{equation*}
(where the first equality follows from equation \eqref{Dr}).
Let $r > s$. By the choice of $\hat{u}, \hat{v} \in \overline{\B}$, we get
\begin{equation*}\label{ug2}
1=\frac{\tilde{d}_2(f(r\B))}{2r}=\max_{u,v \in \overline{\B}} \max _{|q| \leq r}|g_{u,v}(q)| \geq \max_{|q|\leq r}|g_{\hat{u},\hat{v}}(q)|
\ge\max_{|q|\leq s}|g_{\hat{u},\hat{v}}(q)|=1
\end{equation*}
By the Maximum Modulus Principle the function $g_{\hat{u},\hat{v}}$ must be  constant  in $q\in \B$ and equal to $1$ in modulus. 
Consider now $r\in (0,s)$. Then
\[ 1 \geq \frac{\tilde{d}_2(f(r\B))}{2r}= \max_{u,v \in \overline{\B}} \max _{|q| \leq r}|g_{u,v}(q)| \geq \max_{|q|\leq r}|g_{\hat{u},\hat{v}}(q)|=1,\]
which implies,
$
\tilde{d}_2(f(r\B))/2r=1$ for every $r \in (0,1)
$.
The claim is now that $|\partial_c f(0)|=1$.
By \eqref{deriv}, we first of all obtain
\begin{equation}\label{limd2}
\lim_{r \to 0^+} \frac{\tilde{d}_2(f(r\B))}{2r}=1\geq |\partial_c f(0)|.
\end{equation}
Recalling that
\[\frac{\tilde{d}_2(f(r\B))}{2r}=\max_{u,v \in \overline{\B}} \max _{|q| \leq r}|g_{u,v}(q)|,\]
we can get, for every $n \in \mathbb{N}$, the existence of $u_n, v_n \in \overline{\B}$ and $q_n$,
with $|q_n|=\frac{1}{n}$ (converging up to subsequences), such that
\[1=\lim_{n \to \infty} \frac{\tilde{d}_2(f({\frac{1}{n}}\B))}{2\frac{1}{n}}=\lim_{n \to \infty}|g_{u_n,v_n}(q_n)|=|g_{\tilde{u},\tilde{v}}(0)|\leq \max_{u,v \in \overline{\B}}|g_{u,v}(0)|=|\partial_c f(0)|\]
(the last equality is due to \eqref{h0}). A comparison with  \eqref{limd2} concludes the 
proof.
\end{proof}

\section{The $n$-diameter case}
To formulate a $n$-diameter version of the Landau-Toeplitz Theorem for regular functions we begin by giving the definition of $n$-diameter of a subset of $\HH$.  
\begin{defn}\label{n-diametro}
Let $E \subset \HH$. For every $n \in \mathbb{N}$, $n \geq 2$, the {\em $n$-diameter} of $E$ is defined as
\[d_n(E)= \sup_{w_1,...,w_n \in E} \Big( \prod_{1\leq j<k \leq n} |w_k -w_j|\Big)^{\frac{2}{n(n-1)}}.\]
\end{defn}
\noindent As in the complex case (see \cite{Poggi}), we can state
\begin{pro} \label{dn<d2}
For all $n\ge 2$, we have $d_n(E)\le d_2(E)=\Diam(E)$. Moreover  $d_n(E)$ is finite if and only if $d_2(E)$ is finite.  
\end{pro}

\noindent As we did in Section \ref{sec2.1} in the case of the classical diameter $d_2$, we will adopt a specific definition for the $n$-diameter of the image of a subset of $\HH$ under a regular 
function. We will always  consider  images of open balls of the form $r\B$. 
\begin{defn}
Let  $n\geq 2$ and let $f$ be a regular function on $\B$. For $r \in (0,1)$, we
define, in terms of the $*$-product,  
the {\em regular $n$-diameter} of the image of $r\B$ under $f$ as
\[\tilde{d}_n(f(r\B))= \max_{w_1,...,w_n \in \overline{\B}}\, \max_{|q|\leq r} \Big| \stella_{1\leq j < k \leq n}(f_{w_k}(q)-f_{w_j}(q)) 
\Big|^{\frac{2}{n(n-1)}}.\]
Moreover, we define the {\em regular $n$-diameter} of the image of $\B$ under $f$ as
\[\tilde{d}_n(f(\B))=\lim_{r \to 1^{-}}\tilde{d}_n(f(r\B)).\]
\end{defn}
\noindent The same argument used for the regular diameter in Remark \ref{increasing},
guarantees that $\tilde{d}_n(f(\B))$ is well defined.  
Notice that, because of the non-commutativity of quaternions, the order of the factors of a 
$*$-product has its importance. We can choose any order we like, but it has to
be fixed once chosen. In what follows,
when we write $1\leq j<k\leq n$ 
we always mean to order the couples $(j, k)$ with the lexicographic order.
To simplify the notation, we will sometimes  write $j<k$ meaning $1\leq j<k \leq n$.\\
The first step toward understanding the relation between the $n$-diameter and the regular $n$-diameter is the following result.
\begin{pro}\label{Dn<D2}
Let $f:\B\to\HH$ be a regular function, and let $n\geq 2$. Then $\tilde{d}_n(f(\B))\leq \tilde{d}_{2}(f(\B)).$
\end{pro}

\begin{proof} We omit the (technical) proof. The idea is to turn  the $*$-product into an usual product with an iterated application of 
Proposition \ref{trasf}.
\end{proof}
\noindent Notice that Proposition \ref{dn<d2} and Proposition \ref{Dn<D2} imply that if $d_n(f(\B))$ is finite then $\tilde{d}_n(f(\B))$ is finite as well (for any regular function $f$ and $n\geq 2$).

\noindent Let us make some simple remarks about the definition of regular $n$-diameter.
As for the case $n=2$, the regular $n$-diameter is invariant under translation: in fact if $f$ is a regular function on $\B$ and $g$ is defined as $g(q)=f(q)-f(0)$, then $\tilde{d}_n(g(r\B))=\tilde{d}_n(f(r\B))$. Moreover, if $f(q)=qb$ with $b \in \HH$, then $\tilde{d}_n(f(r\B))=|b|d_n(r\B)$; in particular, if $f$ is constant, then $\tilde{d}_n(f(r\B))=0$. Hence if $f$ is of the form $f(q)=a+qb$, for some quaternions $a$ and $b$, then the regular $n$-diameter of $f(r\B)$ coincides with its $n$-diameter.


In order to obtain analogues of inequalities \eqref{LTQ2} and \eqref{LTQ1}, 
in the $n$-diameter case, we study the ratio between the regular $n$-diameter of the image of $r\B$ under a regular function $f$ and the $n$-diameter of the domain $r\B$ of $f$. 
\begin{lem}\label{fi1}
Let $f$ be a regular function on $\B$ and let $n \in \mathbb{N}$, $n \geq 2$. Then 
\[\varphi_n(r)= \frac{\tilde{d}_n(f(r\B))}{d_n(r\B)}=\frac{\tilde{d}_n(f(r\B))}{d_n(\B)r}\]
is an increasing function of $r$ on the open interval $(0,1)$, and 
\[\lim_{r \to 0^+} \varphi_n(r)= |\partial_c f (0)|.\]
\end{lem}
\begin{proof}
If $f$ is a constant or an affine function, then $\varphi_n(r)$ is a constant function. 
So let $f$ be neither constant nor affine.
Fix $w_1,...,w_n \in \overline{\B}$ and consider the auxiliary function
\[ g_{w_1,...,w_n}(q)= d_n(\B)^{-\frac{n(n-1)}{2}} q^{-\frac{n(n-1)}{2}}\stella_{1\leq j< k \leq n}(f_{w_k}(q)-f_{w_j}(q)).\]
Since $f_{w_j}(0)=f(0)$ for every $ j=1,\dots,n ,$
we get that $g_{w_1,\dots,w_n}$ is regular on $\B$. 
\noindent Moreover, using the Maximum Modulus Principle as in \eqref{Dr}, we can write
\begin{equation*}\label{fis}
\varphi_n(r)^{\frac{n(n-1)}{2}}=  \max_{w_1,...,w_n \in \overline{\B}}\max_{|q|\leq r}|g_{w_1,...,w_n}(q)|,
\end{equation*}
and hence we can conclude that $\varphi_n(r)$ is increasing in $r$.

\noindent In turn, to prove the second part of the statement, we proceed as follows. 
\begin{equation*}
\lim_{r \to 0^+} \varphi_n(r)=\lim_{r \to 0^+} \frac{\tilde{d}_n(f(r\B))}{d_n(\B)r}
= \lim_{r \to 0^+}d_n(\B)^{-1}r^{-1}\max_{w_1,...,w_n \in \overline{\B}}\, \max_{|q|\leq r}  \Big| \stella_{j 
< k } (f_{w_k}(q)-f_{w_j}(q))\Big|^{\frac{2}{n(n-1)}}.
\end{equation*}
\begin{equation*}
\begin{aligned}
&\lim_{r \to 0^+} \varphi_n(r)=\lim_{r \to 0^+} d_n(\B)^{-1}r^{-1}\!\!\!\!\!\max_{w_1,...,w_n \in
\overline{\B}}\  \max_{|q|\leq r}  {\Big | \prod_{ j < k }
 (f_{w_k}(T_{j,k}(q))-f_{w_j}(T_{j,k}(q)))\Big|}^{\frac{2}{n(n-1)}}.
\end{aligned}
\end{equation*}
where, for all $j<k$, $T_{j,k}(q)$ is a suitable quaternion belonging to the same sphere $\RRe (q)+|\IIm (q)|\s$ of $q$.
Since for  every $ j<k $ it results
$|T_{k,j}(q)|=|q|$, if $|q|=r$, using the power series expansion of $f$ we can write
\begin{equation*}
\begin{aligned}
&\lim_{r \to 0^+} \varphi_n(r)=\lim_{r \to 0^+} d_n(\B)^{-1}\max_{w_1,...,w_n \in \overline{\B}}\ \max_{|q|=r}  \prod_{j < k}\Big|\sum_{n\geq 1}(T_{k,j}(q))^{n-1} 
(w_k^n-w_j^n) a_n \Big|^{\frac{2}{n(n-1)}}.
\end{aligned}
\end{equation*}
Since $\varphi_n(r)$ is lowerbounded by $0$ and it is increasing in $r$, then the limit
of $\varphi_n(r)$, as $r$ goes to $0$, always exists.
Proceeding as in the proof of Theorem \ref{f0q}, we can find a sequence of points $\{q_m\}_{m\in \N}$, such that $|q_m|=\frac{1}{m}$ for any $m\in\N$, and a sequence of $n$-tuples $\{(w_{1,m},...,w_{n,m})\}_{m\in\N} \subset \overline{\B}^n$, 
converging to some $(\hat{w}_{1},...,\hat{w}_{n}) \in
\overline{\B}^{n}$, such that
\[\lim_{m\to \infty}\varphi_n\left(\frac{1}{m}\right)=d_n(\B)^{-1}\prod_{j < k}\Big|\sum_{n\geq 1}(T_{k,j}(0))^{n-1} (\hat{w}^n_k-\hat{w}^n_j) a_n 
\Big|^{\frac{2}{n(n-1)}}.\]
Therefore, by Definition \ref{n-diametro},  we obtain
\[\lim_{m\to \infty}\varphi_n\left(\frac{1}{m}\right)=d_n(\B)^{-1}|a_1|\prod_{j < k}|(\hat{w}_k-\hat{w}_j)| ^{\frac{2}{n(n-1)}} \leq |a_1|=|\partial_c f(0)|.\]
To prove the opposite inequality, notice that, for every choice of $\{\tilde{w}_1,...,\tilde{w}_n\}\subset
\overline{\B}$, 
\begin{equation*}\label{wtil}
\begin{aligned}
&\lim_{r \to 0^+}\max_{w_1,...,w_n \in \overline{\B}}\ \max_{|q|=r}  \prod_{j < k}\Big|\sum_{n\geq 1}(T_{k,j}(q))^{n-1} (w_k^n-w_j^n) a_n 
\Big|^{\frac{2}{n(n-1)}}\\
&\geq\lim_{r \to 0^+} \max_{|q|=r}  \prod_{j < k}\Big|\sum_{n\geq 1}(T_{k,j}(q))^{n-1} (\tilde{w}_k^n-\tilde{w}_j^n) a_n \Big|^{\frac{2}{n(n-1)}},
\end{aligned}
\end{equation*}
whence
\begin{equation*}
\begin{aligned}
&\lim_{r \to 0^+}\max_{w_1,...,w_n \in \overline{\B}}\ \max_{|q|=r}  \prod_{j < k}\Big|\sum_{n\geq 1}(T_{k,j}(q))^{n-1} (w_k^n-w_j^n) a_n 
\Big|^{\frac{2}{n(n-1)}}\\
&\geq \max_{\tilde w_1,...,\tilde w_n \in \overline{\B}}\ \lim_{r \to 0^+} \max_{|q|=r}  \prod_{j < k}\Big|\sum_{n\geq 1}(T_{k,j}(q))^{n-1} 
(\tilde w_k^n-\tilde w_j^n) a_n 
\Big|^{\frac{2}{n(n-1)}}.
\end{aligned}
\end{equation*}
Therefore we conclude
\begin{equation*}
\begin{aligned}
\lim_{r \to 0^+} \varphi_n(r)\ge
d_n(\B)^{-1}\max_{w_1,...,w_n \in \overline{\B}}\prod_{j < k}|(w_j-w_k) a_1|^{\frac{2}{n(n-1)}}=
|a_1|=|\partial_c f(0)|.
\end{aligned}
\end{equation*}
\end{proof}

\noindent By means of Lemma \ref{fi1}  it is direct to prove the following result.
\begin{teo}\label{Dn}
Let $f$ be a regular function on $\B$ such that $\tilde{d}_n(f(\B))=d_n(\B)$. Then
\begin{equation}\label{n1}
\tilde{d}_n(f(r\B))\leq d_n(r\B) \quad \text{for every} \quad r \in (0,1)
\end{equation}
and
\begin{equation}\label{n2}
|\partial_c f(0)|\leq 1.
\end{equation}
\end{teo}

\noindent We believe that if equality holds in \eqref{n1} for some $r\in(0,1)$ or in \eqref{n2}, then $f$ is affine, but we were not able to prove this statement. On the one hand, it is easy to see that if $f$ is affine, $f(q)=a+qb$ with $a,b \in \HH$, $|b|=1$, 
then equality holds both in \eqref{n1} and in \eqref{n2}; on the other side, we do not
yet know, in general, if the converse holds using the notion of regular $n$-diameter (for $n>2$). 

\section{A $3$-diameter version of the Landau-Toeplitz Theorem} 
In this section we prove a complete $3$-diameter version of the Landau-Toeplitz Theorem. The proof relies upon the elementary fact that
 three points lie always in a same plane. For this reason  the $3$-diameter of a subset 
of $\HH$ (which has dimension  $4$)  is always attained on a bidimensional section of 
the set. 
To compute the $3$-diameter of the unit ball of $\HH$ we need to
recall a preliminary result, about what happens 
in the complex case (for a proof, see e.g. \cite{Poggi}).
Let $\D$ be the open unit disc of $\mathbb{C}$.
\begin{lem}\label{VM}
Given $n$ points $\{ w_1, ..., w_n \} \subset \overline{\D}$ 
\begin{equation*}
\prod_{1\leq j<k\leq n}|w_j -w_k| \leq n^{\frac{n}{2}}.
\end{equation*}
Moreover, equality holds if and only if (after relabeling) $w_j=u \alpha ^j$ with $u \in \mathbb{S}^1$ and $\alpha ^j$  n-th root of unity, i.e. 
$\alpha ^j= e ^{\frac{i2 \pi j }{n}}$, for every $j=1, ..., n$.
\end{lem}

For the $3$-diameter of the unit ball of $\HH$ the following lemma
holds
\begin{lem}\label{radici1}
   For any $I \in \mathbb{S}$ and any $u \in \partial \B$, we have 
$d_3(\B)=\big(|\alpha_2-\alpha_1|| \alpha_3 -\alpha_1|| \alpha_3 -\alpha_2|\big)^{\frac{1}{3}},$
with $\alpha_j=u e^{\frac{I2\pi j}{3}}$, for $j=1,2,3$.
\end{lem}
\begin{proof} The result can be easily proved showing that the $3$-diameter is attained on a maximal disc that, without loss of generality, can be chosen to be some $\B_I$. 

\end{proof}
\noindent Notice that, in particular,  $d_3(\B)=d_3(\D)= \sqrt{3}$.
To prove a $3$-diameter version of the Landau-Toeplitz Theorem for regular
functions, we introduce an appropriate notion of  ``slicewise'' $3$-diameter, inspired by the power series expansion of the regular $3$-diameter. 
\begin{defn}\label{dtetto}
Let $f:\B\to\HH$ be a regular function, and let $\sum_{n\ge 0}q^na_n$ be its power series expansion. 
If $a_N$ is the first non-vanishing coefficient, let us set $\hat{f}$ to be the function obtained by multiplying $f$ (on the right) by $a_N^{-1}|a_N|$,  
\[\hat{f}(q)=\sum_{n\ge 0}q^na_na_N^{-1}|a_N|=\sum_{n\ge 0}q^nb_n,\]
regular on $\B$ as well.
For any $I\in\s$, let $w_1,w_2,w_3$ be points in the closed disc $\overline{\B}_I$, and consider the function
\[\hat g_{w_1,w_2,w_3}(z)=\sum_{n\ge 0}z^n\sum_{k=0}^n\sum_{j=0}^k\big(w_2^j-w_1^{j}\big)\big(w_3^{k-j}-w_1^{k-j}\big)\big(w_3^{n-k}-w_2^{n-k}\big)
b_jb_{k-j}b_{n-k}.\]
holomorphic in all variables $z,w_1,w_2,w_3$ on $\B_I$. 
We define the {\em slice $3$-diameter} of $f(r\B)$ by
\begin{equation}\label{d3slice}
\hat{d}_3(f(r\B))=\sup_{I\in\s}\,\max_{w_1,w_2,w_3 \in \overline{\B}_I}\, \max_{z\in r\overline{\B}} \left|\hat g_{w_1,w_2,w_3}(z) \right|^{1/3},
\end{equation}
and the {\em slice $3$-diameter} of $f(\B)$ as the limit 
\[\hat{d}_3(f(\B))=\lim_{r\to 1^-}\hat{d}_3(f(r\B)).\]
\end{defn}

Thanks to the  Maximum Modulus Principle \ref{PMM}, we get that $r \mapsto \hat{d}_3(f(r\B))$ is an increasing
function, and hence that the previous definition is well posed. 
It is not difficult to prove
that $\hat g_{w_1 ,w_2 ,w_3} (z)$ is continuous as a function of $I$ and of the real and imaginary parts of $z, w_1 , w_2 , w_3$. Hence the
supremum in $I$ appearing in equation \eqref{d3slice} is actually a maximum.
\begin{oss}\label{fdi0}
For any regular function $f:\B \to \HH$, the slice $3$-diameter $\hat{d}_3 (f (\B))$ is the same as the
slice $3$-diameter $\hat{d}_3((f - f (0))(\B))= \hat{d}_3 (f (\B) - f (0))$. Moreover it is easy to prove that if the slice $3$-diameter $\hat d_3 (f (\B))$ 
vanishes, then f is constant.
\end{oss}

In analogy with what happens in the regular $n$-diameter case, we state the following result.
\begin{lem}\label{fitetto}
Let $f$ be a regular function on $\B$, and for $r\in (0,1)$,  let $\hat{\varphi}_3(r)$ be the ratio defined as
\[\hat{\varphi}_3(r)=\frac{\hat{d}_3(f(r\B))}{d_3(r\B)}=\frac{\hat{d}_3(f(r\B))}{d_3(\B)r}.\]
Then $\hat{\varphi}_3(r)$ is increasing in $r$ and 
\[\lim_{r\to 0^+}\hat{\varphi}_3(r)=|\partial_c f(0)|.\]
\end{lem}
\begin{proof} To establish the assertion it is possible to prove that 
\begin{equation*}
\begin{aligned}
\hat{\varphi}_3(r)^3 
=d_3(\B)^{-3} \max_{I\in\s}\max_{w_1,w_2,w_3 \in \overline{\B}_I} \max_{z\in r\overline{\B}_I} \left|z^{-3} \hat g_{w_1,w_2,w_3}(z) \right|
\end{aligned}
\end{equation*}
(see Definition \ref{dtetto}) and use the technique of the proof of Lemma \ref{fi1} on each slice.

\end{proof}

The fundamental tool to prove the ``equality case'' is the following.
\begin{teo}\label{fi.3}
Let $f$ be a regular function on $\B$ and, for $r\in(0,1)$, let 
\[
\hat{\varphi}_3(r)=\frac{\hat{d}_3(f(r\B))}{d_3(\B)r}.
\]
Then $\hat{\varphi}_3(r)$ is strictly increasing in $r$ 
except if $f$ is a constant or an affine function, i.e. if $f(q)=a +qb$ with $a,b \in \HH$.
\end{teo}

\begin{proof}
Thanks to Remark \ref{fdi0} we can suppose $f(0)=0$.
Since $\hat{\varphi}_3(r)$ is increasing for 
$r\in (0,1)$, if it is not strictly increasing, then there exist $s,t$, $0<s<t<1$, such that $\hat{\varphi}_3$ is constant on $[s,t]$.
We will show that this yields that $\hat \varphi_3$ is constant on $(0,t]$.
Let $I\in \s$ and $w_1,w_2, w_3 \in \overline{\B}_I$ be such that
\[ \hat{\varphi}_3(s)^3= d_3(\B)^{-3}\max_{z\in s\overline{\B}_I}\left|z^{-3}\hat g_{w_1,w_2,w_3}(z)\right|.\]
For $r \in [s,t]$, we have $\hat{\varphi}_3(r)=\hat{\varphi}_3(s)$ and, by  
the choice of $w_1,w_2,w_3$,
\[\hat{\varphi}_3(r)^3 \geq d_3(\B)^{-3}\max_{z\in r\overline{\B}_I}\left|z^{-3}\hat g_{w_1,w_2,w_3}(z)\right|\ge d_3(\B)^{-3}\max_{z\in s\overline{\B}_I}\left|z^{-3}\hat g_{w_1,w_2,w_3}(z)\right|=\hat{\varphi}_3(s)^{3}.\] 
Hence, by the  Maximum Modulus Principle \ref{PMM}, we get that the function $z\mapsto z^{-3}\hat g_{w_1,w_2,w_3}(z)$
is constant on $\B_I$.
If we consider now $r \in (0,s)$, then, $\hat{\varphi}_3(r)\le \hat{\varphi}_3(s)$ and
\[\hat{\varphi}_3(r)^{3} 
\geq d_3(\B)^{-3}\max_{z\in r\overline{\B}_I}\left|z^{-3}\hat g_{w_1,w_2,w_3}(z)\right|=
d_3(\B)^{-3}\max_{z\in s\overline{\B}_I}\left|z^{-3}\hat g_{w_1,w_2,w_3}(z)\right|=\hat{\varphi}_3(s)^{3}.\]
Hence $\hat{\varphi}_3(r)=\hat{\varphi}_3(s)$ for all $r \in
(0,t]$. Thanks to Lemma \ref{fitetto}, we obtain then that
\begin{equation*}\label{giulia}
\hat{\varphi}_3(r) \equiv \lim_{r\to 0^+}\hat{\varphi}_3(r)=|\partial_c f(0)|=|a_1|
\end{equation*}
for $r\in [0, t]$. Recalling Remark \ref{fdi0}, we get that either $f$ is constant, or  
$a_1=\de f(0)\neq 0.$
Let us suppose that $f$ is not constant (so that  $b_n=a_na_1^{-1}|a_1|$ for any $n\in \N$).
Recalling the definition of $\hat g_{w_1,w_2,w_3}(z)$, and since the (constant) function $z\mapsto z^{-3}\hat g_{w_1,w_2,w_3}(z)$ 
is equal to its limit at $0$, we have that
\begin{equation*}
\begin{aligned}
|a_1|^3&=\hat\varphi_3^3(r)=d_3(\B)^{-3}
\Big|\sum_{n\ge 3}z^{n-3}\sum_{k=0}^n\sum_{j=0}^k\big(w_2^j-w_1^{j}\big)\big(w_3^{k-j}-w_1^{k-j}\big)\big(w_3^{n-k}-w_2^{n-k}\big)
b_jb_{k-j}b_{n-k}\Big|\\
&=d_3(\B)^{-3} \left|\left(w_2-w_1\right)\left(w_3-w_1\right)\left(w_3-w_2\right)b_1^3\right|\end{aligned}
\end{equation*}
for any $z\in \B_I$. Therefore, thanks to Lemma \ref{radici1}, 
without loss of generality we can suppose that $w_1=1,w_2,w_3$ are cube roots of unity in $L_I$.
Let now  $J$ be an imaginary unit, $J\neq I$, and consider $v_1,v_2,v_3$ cube roots of unity in $L_J$. Then, for any $r\in [0,t]$,
\begin{equation*}
\begin{aligned}
|a_1|&=\hat{\varphi}_3(r)=\!d_3(\B)^{-1} \max_{I\in\s}\max_{w_1,w_2,w_3 \in \overline{\B}_I} \max_{z\in r\overline{\B}_I} \left|z^{-3} \hat g_{w_1,w_2,w_3}(z) \right|^{1/3}\\ 
&\ge d_3(\B)^{-1} \max_{z\in r\overline{\B}_J}\left|z^{-3}\hat g_{v_1,v_2,v_3}(z)\right|^{1/3}\ge d_3(\B)^{-1}  \left|z^{-3}\hat g_{v_1,v_2,v_3}(z)\right|^{1/3}_{z=0}=|a_1|. 
\end{aligned}
\end{equation*} 
Therefore, for any $J\in\s$, if $v_1,v_2,v_3$ are cube roots of unity in $L_J$, the function $z\mapsto z^{-3}\hat g_{v_1,v_2,v_3}(z)\equiv c_J$ is constant on $\B_J$. Notice that $|c_J|$ does not depend on $J\in \s.$
Let now $I$ be an imaginary unit in $\s$, fix $z\in t\B_I$, with $|z|=r$, and let $w_1=1,w_2,w_3$ be cube roots of unity in $L_I$. Consider the function defined for $\zeta\in \B_I$ as
\[h^I_z(\zeta)=z^{-3}\hat g_{\zeta,w_2,w_3}(z)=\sum_{n\ge 3}z^{n-3}\sum_{k=0}^n\sum_{j=0}^k\big(w_2^j-\zeta^{j}\big)\big(w_3^{k-j}-\zeta^{k-j}\big)\big(w_3^{n-k}-w_2^{n-k}\big)
b_jb_{k-j}b_{n-k}.
\]
By construction $\zeta \mapsto h_z^I(\zeta)$ is holomorphic on a neighborhood of $\overline{\B}_I$, and 
\[\left|h^I_z(\zeta)\right|\le  \hat\varphi_3(r)^3d_3(\B)^3=|a_1|^3d_3(\B)^3.\]
Moreover, its value at $\zeta=1$ is
\[h^I_z(1)=z^{-3}g_{1,w_2,w_3}(z)=\left(w_2-1\right)\left(w_3-1\right)\left(w_3-w_2\right)b_1^3=-3\sqrt 3 I |a_1|^3.\] 
Then the function 
\[\zeta \mapsto h^I_z(\zeta)\left(h^I_z(1)\right)^{-1}=h^I_z(\zeta)I(3\sqrt 3)^{-1}|a_1|^{-3}\]
fixes the point $\zeta=1$ and maps the closed unit disc $\overline{\B}_I$ to itself, in fact
\[|h^I_z(\zeta)I(3\sqrt 3)^{-1}|a_1|^{-3}|=|h^I_z(\zeta)|(3\sqrt 3)^{-1}|a_1|^{-3}\le |a_1|^3d_3(\B)^3(3\sqrt 3)^{-1}|a_1|^{-3}=1.\]
We can therefore apply Lemma \ref{Imq2} and we get
\begin{equation*}
\IIm_I\Big(\left.\frac{\p}{\p \zeta}\right|_{\zeta=1} h^I_z(\zeta)I(3\sqrt 3)^{-1}|a_1|^{-3}\Big)=0, \quad \text{that is}\quad 
\RRe\Big(\frac{\p}{\p \zeta}h^I_z(1)\Big)=0.
\end{equation*}
Doing  the same construction for any $J\in\s$, we get that
\begin{equation}\label{re0}
 \RRe\Big(\frac{\p}{\p \zeta}h^J_z(1)\Big)=0
\end{equation}
for any fixed $z\in t\B_J$. An easy computation shows 
\[\frac{\p}{\p \zeta}h^I_z(1)=-\sum_{n\ge 3}z^{n-3}\sum_{k=2}^{n-1}\sum_{j=1}^{k-1}\big(j\big(w_3^{k-j}-1\big)+(k-j)\big(w_2^j-1\big)\big)\big(w_3^{n-k}-w_2^{n-k}\big)
b_jb_{k-j}b_{n-k}.\]
Thanks to the uniform convergence of the series expansion and since equation \eqref{re0} holds for any $z\in t\B_I$, 
we get that the real part of each coefficient must vanish. Namely that, for any $n\in\N$, $n\ge 3$, 
\[\RRe \Big(\sum_{k=2}^{n-1}\sum_{j=1}^{k-1}\big(j\big(w_3^{k-j}-1\big)+(k-j)\big(w_2^j-1\big)\big)\big(w_3^{n-k}-w_2^{n-k}\big)
b_jb_{k-j}b_{n-k}\Big)=0.\] 
The previous equality, together with the fact that it holds true for any $I\in\s$, 
will lead us to conclude that $b_n=a_n\left(a_1^{-1}|a_1|\right)$ is a real number 
for any $n\in\N$. 
In fact, let us proceed by induction. The first step is trivial: $b_0=0$ and
\[b_1=a_1\left(a_1^{-1}|a_1|\right)=|a_1|.\]
Suppose then that $b_1,\dots,b_{s-1}$ are real numbers. 
The first coefficient of the series expansion of $\frac{\p}{\p \zeta}h^I_z(1)$ 
that contains $b_s$ is is the one for which $n=s+2$, and has to satisfy the equation       

\begin{equation}\label{re01}
\RRe\Big(\sum_{k=2}^{s+1}\sum_{j=1}^{k-1}\big(j\big(w_3^{k-j}-1\big)+(k-j)\big(w_2^j-1\big)\big)\big(w_3^{s+2-k}-w_2^{s+2-k}\big)
b_jb_{k-j}b_{s+2-k}\Big)=0.
\end{equation}
In this sum we can gather together the terms containing $b_s$ (which correspond to $(k,j)=(s+1,s),(s+1,1),(2,1)$), that is 
\begin{equation*}
\begin{aligned}
&\left(\left(s\left(w_3-1\right)+\left(w_2^s-1\right)\right)\left(w_3-w_2\right)\right)b_s(b_1)^2+\left(\left(w_3^s-1\right)+s\left(w_2-1\right)\right)\left(w_3-w_2\right)b_s(b_1)^2\\
&+\left(\left(\left(w_3-1\right)+\left(w_2-1\right)\right)\left(w_3^s-w_2^s\right)\right)b_s(b_1)^2\\
&=\big(\sqrt3 I(3s +2 -(w_2^s+w_3^s))-3(w_2^s-w_3^s))\big)|a_1|^2b_s.
\end{aligned}
\end{equation*}
Hence we can split equation \eqref{re01} as 
\begin{equation}\label{Re02}
\begin{aligned}
&\RRe\big(\big(\sqrt3 I(3s +2 -(w_2^s+w_3^s))-3(w_2^s-w_3^s))\big)|a_1|^2b_s\big)\\
&+\RRe \Big(\sum_{j=2}^{s-1}\big(j\big(w_3^{s+1-j}-1\big)+(s+1-j)\big(w_2^j-1\big)\big)\big(w_3-w_2\big)
b_jb_{s+1-j}b_{1}\Big)\\
&+\RRe\Big( \sum_{k=3}^{s}\sum_{j=1}^{k-1}\big(j\big(w_3^{k-j}-1\big)+(k-j)\big(w_2^j-1\big)\big)\big(w_3^{s+2-k}-w_2^{s+2-k}\big)
b_jb_{k-j}b_{s+2-k}\Big)=0. 
\end{aligned}
\end{equation}
Both the second and the third term in equation \eqref{Re02}  do vanish. We will only show this assertion for the second term when $s$ is even. The proofs that the second term with $s$ odd, and that the third term, vanish  are totally similar. If $s$ is even, the second term can be rearranged as
\begin{equation*}
\begin{aligned}
&\RRe\Big(\sum_{j=2}^{s-1}\big(j\big(w_3^{s+1-j}-1\big)+(s+1-j)\big(w_2^j-1\big)\big)\left(w_3-w_2\right)
b_jb_{s+1-j}b_{1}\Big)\\
&=\RRe\Big(
\sum_{j=2}^{s/2}\big(j\big(w_3^{s+1-j}+w_2^{s+1-j}-2\big)+(s+1-j)\big(w_2^j+w_3^j-2\big)\big)\left(w_3-w_2\right)
b_jb_{s+1-j}b_{1}\Big)\\
\end{aligned}
\end{equation*}
Since $w_2^n+w_3^n\in \rr$ and $w_2^n-w_3^n\in I\rr$ for any $n\in\N$, and $b_n\in \rr$ for any $n=1,\dots,s-1$, we get that
\[\RRe\Big(
\sum_{j=2}^{s-1}\big(j\big(w_3^{s+1-j}+w_2^{s+1-j}-2\big)+(s+1-j)\big(w_2^j+w_3^j-2\big)\big)\big(w_3-w_2\big)
b_jb_{s+1-j}b_{1}\Big)=0.\]
Then equation \eqref{Re02} reduces to 
$\RRe\big(\big(\sqrt3 I(3s +2 -(w_2^s+w_3^s))-3(w_2^s-w_3^s))\big)|a_1|^2b_s\big)=0$. Therefore, for any $s \in \N$, there exists $\alpha_s\in \rr$ such that $\RRe(\alpha_sIb_s)=\alpha_s\RRe(Ib_s)=\IIm_I(b_s)=0$ for all $I\in\s$.
Hence we get $b_s\in\rr$ for all $s$.
Recalling that $b_n$ are the coefficients of the power series of $\hat{f}$, 
we get that $\hat{f}(\B_I)\subseteq L_I$ for all $I\in \s$, and hence  $\hat{f}$ is complex holomorphic on each slice.
Observe that for any $r\in(0,t]$ the slice $3$-diameter of $f(r\B)$ coincides with the usual $3$-diameter of $\hat{f}(r\B)$. In fact, for any $I\in\s$, 
\begin{equation*}
\begin{aligned}
&\hat g_{w_1,w_2,w_3}(z)=\sum_{n\ge 0}z^n\sum_{k=0}^n\sum_{j=0}^k\big(w_2^j-w_1^{j}\big)\big(w_3^{k-j}-w_1^{k-j}\big)\big(w_3^{n-k}-w_2^{n-k}\big)
b_jb_{k-j}b_{n-k}\\
&=\Big(\sum_{n\ge 0}\left((zw_2)^n-(zw_1)^n\right)b_n\Big)\Big(\sum_{n\ge 0}\left((zw_3)^n-(zw_1)^n\right)b_n\Big)\Big(\sum_{n\ge 0}\left((zw_3)^n-(zw_2)^n\right)b_n\Big)\\
&=\big(\hat{f}(zw_2)-\hat{f}(zw_1)\big)\big(\hat{f}(zw_3)-\hat{f}(zw_1)\big)\big(\hat{f}(zw_3)-\hat{f}(zw_2)\big).
\end{aligned}
\end{equation*}
Hence, for all $I\in\s$,
\begin{equation*}
\begin{aligned}
\hat{d}_3(f(r\B))&=\max_{w_1,w_2,w_3\in \overline{\B}_I}\max_{z\in r\overline{\B}_I}\big|\hat g_{w_1,w_2,w_3}(z)\big|^{1/3}\\
&=\max_{w_1,w_2,w_3\in \overline{\B}_I}\max_{z\in r\overline{\B}_I}\big|\big(\hat{f}(zw_2)-\hat{f}(zw_1)\big)\big(\hat{f}(zw_3)-\hat{f}(zw_1)\big)\big(\hat{f}(zw_3)-\hat{f}(zw_2)\big)\big|^{1/3}\\
&=d_3(\hat{f}(r\B_I)).
\end{aligned}
\end{equation*}
Thanks to the complex $n$-diameter version of the Landau-Toeplitz Theorem, \cite{Poggi}, we obtain that $\hat f$ is an affine function, $\hat f(q)=b_0 + q b_1 = a_0a_1^{-1}|a_1|+ q|a_1|$
and hence that $f$ is affine as well, $f(q)=a_0+qa_1$.  

\end{proof}
We can finally prove the 
\begin{teo}[Landau-Toeplitz Theorem for the slice $3$-diameter]
Let $f$ be a regular function on $\B$ such that $\hat{d}_3 f(\B)=d_3(\B)$. Then
\begin{equation}\label{d31}
\hat{d}_3 (f(r\B))\leq d_3(r\B) \quad \text{for every} \quad r\in (0,1)
\end{equation}
and
\begin{equation}\label{d32}
|\partial_c f(0)|\leq 1.
\end{equation}
Moreover equality holds in \eqref{d31}, fore some $r \in (0,1)$, or in \eqref{d32}, if and only if $f$ is an affine function, $f(q)=a + qb$ with 
$a,b \in \HH$, and $|b|=1$.
\end{teo}

\begin{proof}
By Lemma \ref{fitetto}, both inequalities hold true. For the equality case, if $f(q)=a + qb$ with $a,b \in \HH$, $|b|=1$, it is easy to see that 
equality holds in both statements.
Otherwise, if equality holds in \eqref{d31} or in \eqref{d32}, then $\hat{\varphi}_3(r)$ defined in Lemma \ref{fitetto} 
is not strictly increasing. Theorem \ref{fi.3} implies then that $f$ is 
an affine function. Moreover, since $\hat{d}_3 ((f(\B))=d_3(\B)$, the coefficient of the first degree term of $f$ has unitary modulus. 
\end{proof}

Notice that the notion of the slice $3$-diameter does not make sense for $n\geq 4$. Moreover the $n$-diameter of $\B$, when $n \geq 4$, is not anymore attained 
at points that lie on a same plane $L_I$. In fact, the following result holds true. 
\begin{pro} For all $I\in \mathbb{S}$  the inequality $d_4(\B)>d_4(\B_I)$ holds.
\end{pro}

\begin{proof}
The proof follows from the direct computation of the $4$-diameter of a maximal tetrahedron contained in $\B$. 

\end{proof}
The proof of Theorem \ref{fi.3} heavily relies upon the fact that both the $3$-diameter of $\B$ and the slice $3$-diameter of $f(\B)$ are attained at a complete set of cube roots of unity
lying on a same plane $L_I$. We have no alternative proof to use when $n\geq 4$.

\end{document}